\documentclass[leqno,12pt]{amsart} %leqno is the option to put formula numbers on the left side
\usepackage{amssymb, amsmath, stmaryrd, mathtools, thmtools, thm-restate, bm}
\usepackage{tikz-cd}
\usepackage{stackengine}
\theoremstyle{plain} %text of this environment is typesetted in italics
\newtheorem{theorem}{\indent\sc Theorem}[section]
\newtheorem{lemma}[theorem]{\indent\sc Lemma}
\newtheorem{corollary}[theorem]{\indent\sc Corollary}

\theoremstyle{definition} %text of this environment is typesetted in roman letters
\newtheorem{definition}[theorem]{\indent\sc Definition}
\newtheorem{remark}[theorem]{\indent\sc Remark}

\newtheorem*{jacconjecture*}{\indent\sc  \(\text{JC}(R, n)\)}
\newtheorem*{tjconjecture*}{\indent\sc \(\text{TJC}(R, I, n)\)}
\newtheorem*{unimodularconjecture}{\indent\sc Unimodular Conjecture for \(\mathbb{Z}_p\)}
\DeclareMathOperator{\ZZ}{\mathbb{Z}}
\DeclareMathOperator{\CC}{\mathbb{C}}

\DeclareMathOperator{\QQ}{\mathbb{Q}}

\DeclareMathOperator{\JC}{JC}
\DeclareMathOperator{\TJC}{TJC}

\begin{document}

\title[A Tate algebra version of the Jacobian conjecture]{A Tate algebra version of the Jacobian conjecture} %title of paper and the running head option

\author[Lucas Hamada]{Lucas Hamada} %first author's name and the running head option

\author[Kazuki Kato]{Kazuki Kato} %second author's name and the running head option

\author[Ryo Komiya]{Ryo Komiya}

%\dedicatory{Dedicated to Professor Xxx Yyy on his sixtieth birthday}

%%%%%%%%%%%%%%% footnote %%%%%%%%%%%%%%%%
%In case \subjclass[2020] command is not effective
%(or the version of amsart.cls is old), write as follows instead:
%\renewcommand{\thefootnote}{\fnsymbol{footnote}}
%\footnote[0]{2020\textit{ Mathematics Subject Classification}.
%Primary 00; Secondary 00.}
%
%%%%%%%%%%%% Authors' addresses %%%%%%%%%%%%%
\address{% First Author
Department of Mathematics \endgraf
Institute of Science Tokyo \endgraf
2-12-1, O-okayama, Meguro-ku, Tokyo 152-8551 \endgraf
Japan
}
\email{lucas.h.r.hamada@gmail.com}

\address{% Second Author
Department of Mathematics \endgraf
Institute of Science Tokyo \endgraf
2-12-1, O-okayama, Meguro-ku, Tokyo 152-8551 \endgraf
Japan
}
\email{kato.k.ba@m.titech.ac.jp}

\address{% Third Author
Department of Mathematics \endgraf
Institute of Science Tokyo \endgraf
2-12-1, O-okayama, Meguro-ku, Tokyo 152-8551 \endgraf
Japan
}
\email{komiya.r.ab@m.titech.ac.jp}
%%%%%%%%%%%%%%%%%%%%%%%%%%%%%%%%%%%%%%%%%

\maketitle

\begin{abstract}
    This paper investigates a Tate algebra version of the Jacobian conjecture, referred to as {\it the Tate-Jacobian conjecture}, for commutative rings \(R\) equipped with an \(I\)-adic topology. We show that if the \(I\)-adic topology on \(R\) is Hausdorff and \(R/I\) is a subring of a \(\QQ\)-algebra, then the Tate-Jacobian conjecture is equivalent to the Jacobian conjecture. Conversely, if \(R/I\) has positive characteristic, the Tate-Jacobian conjecture fails. Furthermore, we establish that the Jacobian conjecture for \(\CC\) is equivalent to the following statement: for all but finitely many primes \(p\), the inverse of a polynomial map over \(\CC_p\) whose Jacobian determinant is an element of \(\CC_p^\times\) lies in the Tate algebra over \(\CC_p\).
\end{abstract}

\section*{Introduction}

The {\bf Jacobian conjecture} (JC), first proposed by O.-H. Keller in 1939 \cite{Keller} and later included among Smale's problems, concerns the invertibility of polynomial maps over a commutative ring \(R\) under the assumption that their Jacobian determinant is a unit in the polynomial ring. Given a polynomial map  
\(
F = (f_1, \ldots, f_n) \in R[X_1, \ldots, X_n]^n,
\)  
its Jacobian matrix is defined by
\[
JF = \left( \frac{\partial f_i}{\partial X_j} \right)_{1 \leq i,j \leq n}.
\]  
The conjecture, referred to as \(\JC(R,n)\), asserts the following:  

\begin{align*}
\begin{split}
    &\text{If } F \in R[X_1, \ldots, X_n]^n\text{ satisfies }\det JF \in R[X_1, \ldots, X_n]^\times,\text{ then there exists }\\&G \in R[X_1, \ldots, X_n]^n\text{ such that }F \circ G = G \circ F = (X_1, \ldots, X_n).
\end{split}
\tag*{\(\JC(R,n)\)}
\end{align*} 

\noindent
 In this paper, \(\JC(R)\) refers to the conjunction of \(\JC(R, n)\) for all \(n \geq 1\). The existing results on the Jacobian conjecture depend on the characteristic of the ring \( R \). Here, \( R \) is said to have characteristic \( c \), where \(c\) is a nonnegative integer, if the kernel of the unique homomorphism \( \ZZ \to R \) is the ideal \( c\ZZ \).

 \begin{itemize}
     \item {\it Characteristic zero case}: E. Connell and L. van den Dries \cite{CD}, together with H. Bass, E. Mutchnik, and D. Wright \cite{BMW}, proved that for any integral domains \( R \) and \( S \) of characteristic zero, \( \JC(R) \) is equivalent to \( \JC(S) \). Later, A. van den Essen \cite[Theorem 1.1.12]{Es} generalized this result to any rings \( R \) and \( S \) that are subrings of \( \mathbb{Q} \)-algebras.
     \item {\it Positive characteristic case}: For a ring \( R \) of characteristic \( c > 0 \), \( \JC(R, n) \), and hence \(\JC(R)\), is known to be false. For instance, the polynomial map \(F = (X_1 - X_1^c, \ldots, X_n - X_n^c)\) provides a counterexample.
 \end{itemize}  

This paper defines and explores the {\bf Tate-Jacobian Conjecture} (TJC), a variation of the Jacobian conjecture in the context of Tate algebras over rings equipped with a topology induced by an ideal. Let \( I \) be an ideal of a ring \( R \). The Tate algebra \( (R, I) \langle X_1, \ldots, X_n \rangle \) consists of formal power series over \( R \) whose coefficients converge to zero in the \( I \)-adic topology on \( R \). We study the following conjecture, referred to as \( \TJC(R, I, n) \), for rings \( R \) with an ideal \( I \) and integers \( n \geq 1 \):

\begin{align}
\begin{split}
    &\text{If }F \in (R, I)\left<X_1, \ldots, X_n\right>^n\text{ satisfies }\det JF \in (R, I)\left<X_1, \ldots, X_n\right>^\times\\&\text{and }F(0) = 0\text{, then there exists }G \in (R, I)\left<X_1, \ldots, X_n\right>^n\text{ such that}\\&F \circ G = G \circ F = (X_1, \ldots, X_n).
\end{split}
\tag*{\(\TJC(R,I,n)\)}
\end{align} 

\noindent
Similarly to the Jacobian conjecture, \( \TJC(R, I) \) refers to the conjunction of \( \TJC(R, I, n) \) for all \( n \geq 1 \). The following cases illustrate its scope:

    \begin{itemize} 
        \item If \(I\) is nilpotent, that is, \(I^e = (0)\) for some \(e \geq 1\), then \(\TJC(R, I, n)\) is equivalent to \(\JC(R,n)\). 
        \item If \(I = R\), then \(\TJC(R, I)\) corresponds to the analog of the Jacobian conjecture for formal power series, which is known to be true (see \cite[Theorem 1.1.2]{Es}).
    \end{itemize}

From the cases above, we see that the Tate-Jacobian conjecture generalizes the Jacobian conjecture. When the \(I\)-adic topology on \(R\) is Hausdorff, we obtain an equivalence between the two conjectures.   

\begin{restatable}{theorem}{equichar}\label{thm: TJ conjecture for char zero}
Let \(R\) be a ring with an ideal \(I\). If \(\TJC(R,I,n)\) holds, then \(\JC(R/I,n)\) holds. Conversely, if the \(I\)-adic topology on \(R\) is Hausdorff, then \(\JC(R/I,n)\) implies \(\TJC(R,I,n)\).
\end{restatable}

By the previous theorem, together with the results from \cite{BMW, CD, Es}, we obtain the following corollary:

\begin{corollary}\label{cor: TJ conjecture for char zero}
    The following conditions are equivalent:
    \begin{enumerate}
        \item[(i)] \(\TJC(R, I)\) holds for some ring \( R \) with an ideal \( I \) such that the \( I \)-adic topology on \( R \) is Hausdorff and \( R/I \) is a subring of a \( \QQ \)-algebra.
        \item[(ii)] \(\JC(R)\) holds for some ring \( R \) which is a subring of a \( \QQ \)-algebra.
        \item[(iii)] \(\JC(R)\) holds for all rings \( R \) which are subrings of a \( \QQ \)-algebra.
    \end{enumerate}
\end{corollary}

\noindent
Additionally, when \(R/I\) has positive characteristic, we obtain the following corollary:

\begin{corollary}\label{thm: TJ conjecture for char > 0}
    Let \(R\) be a ring and \(I\) an ideal such that \(R/I\) has characteristic \(c > 0\). Then, \(\TJC(R, I, n)\) is false.
\end{corollary}

Hence, the Tate-Jacobian conjecture fails for important rings such as the \( p \)-adic rings. However, by refining the assumptions of the Tate-Jacobian conjecture and applying a result from A. van den Essen and R. J. Lipton \cite{EsL}, we can establish an equivalence for the Jacobian conjecture within the framework of Tate algebras over \( p \)-adic rings:

\begin{restatable}{theorem}{lasttheorem}\label{main theorem 1}
The Jacobian conjecture for \(\CC\) is equivalent to the following statement: for every \(n \geq 1\) and for all but finitely many primes \(p\), if \(F \in \CC_p[X_1, \ldots, X_n]^n\) satisfies \(\det JF \in \CC_p^\times\), then there exists \(G \in \CC_p\left\langle X_1, \ldots, X_n\right\rangle^n\) such that \(F \circ G = G \circ F = (X_1, \ldots, X_n)\).
\end{restatable}

\subsection*{Notation and Conventions:}
\begin{itemize}
    \item All rings are commutative and have a nonzero identity element.
    \item The notation \(X\) usually refers to the \(n\) variables \(X_1, \ldots, X_n\). For example, \(R[X]\) denotes the ring of polynomials in \(n\) variables \(R[X_1, \ldots, X_n]\) (see Section \ref{section: On the Tate-Jacobian conjecture}).
    \item If the ideal \(I\) defining the \(I\)-adic topology on \(R\) is clear from context, we write the Tate algebra \((R, I)\left<X_1, \ldots, X_n \right>\) simply as \(R\left<X_1, \ldots, X_n\right>\) (see Section \ref{section: Generalities of Tate Algebras}).
\end{itemize}

\subsection*{Outline of the Paper}  
Section \ref{section: Generalities of Tate Algebras} provides a brief overview of the definition and notation of Tate algebras. In Section \ref{section: On the Tate-Jacobian conjecture}, we study the Tate-Jacobian conjecture and prove Theorem \ref{thm: TJ conjecture for char zero}. Finally, Section \ref{section: Equivalence of Jacobian conjecture for mixed char} presents the proof of Theorem \ref{main theorem 1}.

\subsection*{Acknowledgments}

The authors would like to express their gratitude to Yuichiro Taguchi and Naganori Yamaguchi for their valuable feedback and insightful advice on the first draft of this paper.

\section{Generalities on Tate Algebras}\label{section: Generalities of Tate Algebras}

In this section, we recall the definition and the characterization of the unit group of Tate algebras over topological rings. For a more detailed account, we refer the reader to \cite{BGR}.  

\begin{definition}
    Let \(R\) be a topological ring. The \textit{Tate algebra} over \(R\) in \(n\) variables, denoted by \( R\left\langle  X_1, \ldots, X_n \right\rangle \), is the subalgebra of the \(R\)-algebra of formal power series \( R\llbracket X_1, \ldots, X_n \rrbracket \) consisting of power series  
    \[
    f(X_1, \ldots, X_n) = \sum_{\nu_1, \ldots, \nu_n \geq 0} a_{\nu_1, \ldots, \nu_n} X_1^{\nu_1} \cdots X_n^{\nu_n} \in R\llbracket X_1, \ldots, X_n \rrbracket
    \]
    such that the coefficients \(a_{\nu_1, \ldots, \nu_n}\) tends to zero in the topology of \(R\). Specifically, let \(\{\mathcal{U}_i\}_{i \in I}\) be a system of neighborhood of \(0\) in the topology of \(R\). Then, for every \(i \in I\), we have \(a_{\nu_1, \ldots, \nu_n} \in \mathcal{U}_i\) for all but finitely many indices \((\nu_1, \ldots, \nu_n)\).  
\end{definition}

\begin{remark}
    Let \(R\) be a ring with an ideal \(I\). We denote by \((R,I)\left<X_1, \ldots, X_n\right>\) the Tate algebra over the ring \(R\) with the \(I\)-adic topology, and simply write \(R\left<X_1, \ldots, X_n\right>\) when \(I\) is clear from context. 
\end{remark}

When \(R\) is equipped with the topology induced by an ideal \(I\), we can determined the group of units \((R,I)\left<X_1, \ldots, X_n\right>^\times\). This is showed by the following criterion:

\begin{lemma}[Adapted from Proposition 1, Chapter 5, \cite{BGR}]\label{lemma: unit group of tate algebra}
    A series \( f \in (R,I)\left\langle  X_1, \ldots, X_n \right\rangle \) is a unit if and only if \( f(0) \in R^\times \) and \( f - f(0) \) has all its coefficients in the radical \(\sqrt{I}\).
\end{lemma}

\begin{proof}
    Suppose \(f\) is a unit in \((R,I)\left\langle X_1, \ldots, X_n \right\rangle\). Then, \(f\) must also be a unit in the ring of formal power series \(R\llbracket X_1, \ldots, X_n \rrbracket\). This implies \(f(0) \in R^\times\). Furthermore, the reduction \(\overline{f} \in (R/I)[X_1, \ldots, X_n]\) is a unit, which forces all the coefficients of \(f - f(0)\) to belong in \(\sqrt{I}\). Conversely, assume \(f(0) \in R^\times\) and that \( f - f(0) \) has all its coefficients in \(\sqrt{I}\). Let \(b = f(0)^{-1} \in R\). Then, we can write \(bf = 1 + u\), where \(u(0) = 0\), \(u \in (R,I)\left\langle X_1, \ldots, X_n \right\rangle\), and the coefficients of \(u\) are in \(\sqrt{I}\). Now, we claim that the sum \( \sum_{i\geq 0}(-u)^i \), which is well-defined since \(u(0) = 0\), belongs to \((R, I)\left\langle X_1, \ldots, X_n \right\rangle\). Indeed, since \(u\) has only finitely many coefficients in \(\sqrt{I} \setminus I\), some power of \(u\) must have all its coefficients in \(I\).  So for every positive integers $N$, there exists $a_N \in \ZZ$ such that $u^{a_N}$ has all its coefficients in $I^N$. Since \( u^i \in (R,I)\langle X_1,\ldots,X_n\rangle \), there exists a positive integer \( A^{(N)}_i \) such that if \( \nu_1 + \cdots + \nu_n > A^{(N)}_i \), then the coefficient \( a_{\nu_1,\ldots,\nu_n} \) belongs to \( I^N \).
     Now, defining 
    \[
      A^{(N)} := \max_{1\leq i \leq a_N} A^{(N)}_i,
    \]
    it follows that for every \( u^i \), the coefficient of any monomial \( X^{\nu_1} \cdots X^{\nu_n} \) satisfying \( \nu_1 + \cdots +\nu_n > A^{(N)} \) is necessarily in \( I^N \). Consequently, the same holds for the sum \( \sum_{i\geq 0}(-u)^i \). Therefore, the product \(b \sum_{i \geq 0} (-u)^i\) is the inverse of \(f\) in \((R,I)\left\langle X_1, \ldots, X_n \right\rangle\). 
\end{proof}

\section{On the Tate-Jacobian conjecture}\label{section: On the Tate-Jacobian conjecture}

In this section, we study the Tate-Jacobian conjecture for rings equipped with a topology induced by an ideal. As stated in the introduction, we denote the \(n\) variables \(X_1, \ldots, X_n\) simply by \(X\). 

  \equichar*

  The key part of the proof is the following lemma, which relates the existence of the inverse of \(F \in R\left<X\right>^n\) to the existence of the inverse of its projection \(\overline{F} \in (R/I)[X]^n\).

\begin{lemma}\label{Lemma inverse projection}
     Let \(R\) be a ring and \(I\) an ideal such that the \(I\)-adic topology on \(R\) is Hausdorff. For any \(F \in R\left\langle  X\right\rangle^n\) such that \(F(0) = 0\), the following equivalence holds\(:\) \(F\) is invertible in \(R\left\langle  X\right\rangle^n\) if and only if its projection \(\overline{F}\) is invertible in \((R/I)[X]^n\).
\end{lemma}

\begin{proof}
    To prove the equivalence, we first assume that \(F\) is invertible in \(R\left\langle  X\right\rangle^n\), and prove that \(\overline{F}\) is invertible in \((R/I)[X]^n\). By assumption, there exists \(G \in R\left\langle  X\right\rangle^n\) such that \(F \circ G = G \circ F = X\). Consider the projection \(\overline{G}\) of \(G\) in \((R/I)[X]^n\). Then, we see that
    \[
        X = \overline{G \circ F} = \overline{G} \circ \overline{F},
    \]
    and conclude that \(\overline{G}\) is the inverse of \(\overline{F}\) in \((R/I)[X]^n\).

    Next, we show the reverse implication. Suppose that \(\overline{F}\) is invertible in \((R/I)[X]^n\), and let \(\overline{G} = (\overline{G_1}, \ldots, \overline{G_n})\) be its inverse. Let \(G_i\) be a lift of \(\overline{G_i}\) to \(R[X]^n\) such that \(G_i(0) = 0\) for every \(i\). Then, 
    \[
        G_i(F(X)) - X_i \in I\!R\left\langle X\right\rangle,
    \]
    where \(I\!R\left\langle X\right\rangle\) denotes the subset of \(R\left< X\right>\) of power series whose of which coefficients belong to the ideal \(I\). Hence, there exists an \(f \in I\!R\left\langle X\right\rangle\) such that \(f(0) = 0\) and 
    \[
        X_i = G_i(F(X)) + f \in R\left<F\right> + I\!R\left\langle X\right\rangle,
    \]
    where \(R\left<F\right>\) is the abbreviation of \(R\left<F_1, \ldots, F_n\right>\).
    Since the choice of \(i\) was arbitrary, and the constant term of \(G_i(F(X)) + f\) is zero, we obtain \(R\left\langle X\right\rangle \subseteq R\left\langle F\right\rangle + I\!R\left\langle X\right\rangle\). Therefore, 
    \[
        R\left\langle X\right\rangle \subseteq R\left\langle F\right\rangle + I\!R\left\langle X\right\rangle \subseteq R\left\langle F\right\rangle + I(R\left\langle F\right\rangle + I\!R\left\langle X\right\rangle) \subseteq (R\left\langle F\right\rangle + I\!R\left\langle F\right\rangle) + I^2\!R\left\langle X\right\rangle.
    \]
    By repeating this process for any natural number \(e\) we obtain:
    \[
        R\left\langle X\right\rangle \subseteq (R\left\langle F\right\rangle + I\!R\left\langle F\right\rangle + \cdots + I^{e-1}\!R\left\langle F\right\rangle) + I^e\!R\left\langle X\right\rangle \subseteq R\left\langle F\right\rangle + I^e\!R\left\langle X\right\rangle.
    \]
    Since the \(I\)-adic topology is Hausdorff, it satisfies \(\cap_{i \geq 1}I^i = (0)\). We conclude that \(R\left\langle X\right\rangle \subseteq R\left\langle F\right\rangle\), and therefore \(R\left\langle F\right\rangle = R\left\langle X\right\rangle\). This shows that \(F\) is invertible in \(R\left\langle X\right\rangle^n\).
\end{proof}

We now proceed to prove Theorem \ref{thm: TJ conjecture for char zero}.

\begin{proof}[Proof of Theorem \ref{thm: TJ conjecture for char zero}]
    First, suppose \(\TJC(R, I, n)\) holds, and let \(F \in (R/I)[X]^n\) be such that \(\det JF \in (R/I)[X]^\times\). We claim that we may assume, without loss of generality, that \(F = X + (\text{degree \(\geq 2\) terms})\). Indeed, let \( F_n \) denote the homogeneous component of \( F \) of degree \(n\), so that \( F \) can be decomposed as  \( F = F_0 + F_1 + F_2 + \cdots \). The Jacobian matrix \( JF \) can similarly be expressed as  \( JF = JF_1 + JF_2 + \cdots \) where each \( JF_n \) is a matrix of homogeneous polynomials of degree \( n-1 \). Since \( \det JF \in (R/I)[X]^\times \), it follows that  \(\det JF_1 = \det JF(0) \in (R/I)^\times\). As \( F_1 \) is linear, there exists a linear map \( G' \in (R/I)[X]^n \) such that \(G' \circ F_1 = X\). Define \[F' := G' \circ (X - F(0)) \circ F = X + (\text{degree \(\geq 2\) terms}). \] If \( F' \) has an inverse map \( G \), then the composition \(G \circ G' \circ (X - F(0))\) is clearly the inverse map of \( F \).
    
    Next, choose a lift \(\widetilde{F} = X + (\text{degree \(\geq 2\) terms}) \in R[X]^n\) of \(F\). Then, we have
    \[
    \det J\widetilde{F} = 1 + (\text{degree \(\geq 1\) terms}),
    \]
    and 
    \[\overline{\det J\widetilde{F}} = \det J\overline{\widetilde{F}} \in (R/I)[X]^\times.
    \]
    Hence, \(\det J\widetilde{F}(0) = 1 \in R^\times\) and \(\det J\widetilde{F} - \det J\widetilde{F}(0)\) has coefficients in \(\sqrt{I}\). By Lemma \ref{lemma: unit group of tate algebra}, we deduce that \(\det J\widetilde{F} \in (R,I)\left<X\right>^\times\) and hence there exists \(\widetilde{G} \in R\left\langle X \right\rangle^n\) such that \(\widetilde{F} \circ \widetilde{G} = \widetilde{G} \circ \widetilde{F} = X\). Therefore, the image \(G\) of \(\widetilde{G}\) in \((R/I)[X]^n\) is the inverse of \(F\).

    Conversely, suppose the \(I\)-adic topology on the ring \(R\) is Hausdorff and \(\JC(R/I, n)\) holds. Let \(F \in R\left\langle X \right\rangle^n\) be such that \(F(0) = 0\) and \(\det JF \in R\left\langle X \right\rangle^\times\), and denote by \(\overline{F}\) the image of \(F\) in \((R/I)[X]^n\). Since \(\det J\overline{F} \in (R/I)[X]^\times\), by assumption, \(\overline{F}\) is invertible in \((R/I)[X]^n\). Thus, by Lemma \ref{Lemma inverse projection}, \(F\) is invertible in \(R\left\langle X \right\rangle^n\).
\end{proof}

\section{Proof of Theorem \ref{main theorem 1}}\label{section: Equivalence of Jacobian conjecture for mixed char}

We conclude this paper with a proof of Theorem \ref{main theorem 1}. The main part of the proof was already established by A. van den Essen and R. J. Lipton \cite{EsL}, who showed that the Jacobian conjecture for \(\CC\) holds if and only if the Unimodular conjecture for \(\ZZ_p\) holds for all but finitely many primes \(p\):

 \begin{unimodularconjecture}
    {\it For every \(n \geq 1\), and for every \(F \in \ZZ_p[X]^n\) satisfying \(\det JF \in \ZZ_p^\times\), there exists \(b \in \ZZ_p^n\) such that \(F(b) \in \ZZ_p^n\) is unimodular, that is, the ideal generated by the entries of the vector \(F(b)\) is \(\ZZ_p\).}
 \end{unimodularconjecture}

 The idea of our proof is to relate Theorem \ref{main theorem 1} to the Unimodular conjecture. To proceed, we first establish a key result that enables us to reduce the problem to \( \ZZ_p \):

 \begin{lemma}\label{lemma restriction of the inverse} 
    Let \( R \) and \( S \) be topological rings such that \( R \subseteq S \) and \( R \) has the subspace topology induced from \( S \). Let \( F \in R\left\langle X \right\rangle^n \) be such that \( F(0) = 0 \) and \( \det JF(0) \in R^\times \). If there exists \( G \in S\left\langle X \right\rangle^n \) such that \( F \circ G = G \circ F = X \), then \( G \in R\left\langle X \right\rangle^n \).
\end{lemma}

\begin{proof}
    By \cite[Theorem 1.1.2]{Es}, \( F \) has a unique inverse \( H \) in \( R\llbracket X \rrbracket^n \), which is also the unique inverse of \( F \) in \( S\llbracket X \rrbracket^n \). Since \( G \) is also an inverse of \( F \) in \( S\left\langle X \right\rangle^n \subseteq S\llbracket X \rrbracket^n \), the uniqueness of the inverse implies that \( G = H \).  
    Therefore, the inverse of \( F \) must belong to the intersection
    \[
    R\llbracket X\rrbracket^n \cap S\left\langle X \right\rangle^n.
    \]
    By the definition of the Tate algebra, and since \( R \) has the subspace topology induced from \( S \), this intersection is precisely \( R\left\langle X \right\rangle^n \).
\end{proof}

\begin{proof}[Proof of Theorem \ref{main theorem 1}]
    For simplicity, we denote the following proposition by (\(\star\)):
    \begin{align}
\begin{split}
    &\textit{For every }n \geq 0\textit{ and for all but finitely many primes }p,\textit{ if }F \in \CC_p[X]^n\textit{ satisfies}\\&\det JF \in \CC_p^\times\textit{, then there exists }G \in \CC_p\left\langle  X\right\rangle^n\textit{ such that }F \circ G = G \circ F = X.
\end{split}
\tag{\(\star\)}
\end{align} 

    First, suppose that \(\JC(\CC)\) holds. By the results of \cite{BMW, CD}, \(\JC(\CC_p)\) holds for all primes \(p\), which clearly implies (\(\star\)). 

    Conversely, suppose that (\(\star\)) holds, and let \(F \in \ZZ_p[X]^n\) be such that \(\det JF \in \ZZ_p^\times\). By a suitable translation, we assume without loss of generality that \(F(0) = 0\).Under this assumption, we can apply Lemma \ref{lemma restriction of the inverse} and conclude the existence of \(G \in \ZZ_p\left<X \right>^n\) such that \(F \circ G = G \circ F = X\). Observe that \(G\) converges at every point \(x \in \ZZ_p^n\), in particular, it converges at \(\bm{1} = (1, \ldots, 1)\). Hence, if we set \(b = G(\bm{1})\), then \(F(b) = (1, \ldots, 1)\) is unimodular. Since the choice of \(n \geq 1\) and the prime \(p\) were arbitrary, we conclude that the Unimodular conjecture for \(\ZZ_p\) holds for all but finitely many primes \(p\). Therefore, by \cite[Theorem 6]{EsL}, \(\JC(\CC)\) holds.
\end{proof}

%%%%%%%%%%%% References %%%%%%%%%%%%%

\end{document}